\documentclass[12pt,a4paper,oneside]{amsart}

\usepackage{amsthm,enumerate,verbatim,amssymb,color,esint,hyperref}
\usepackage[left=2cm,right=3.5cm,top=2cm,bottom=2cm]{geometry}

\usepackage[usenames,dvipsnames]{xcolor}
\usepackage{pgf, tikz}

\usetikzlibrary{arrows}
\usetikzlibrary{svg.path}
\usepackage{graphicx}
\usepackage{pgffor,pgfmath}
\usetikzlibrary{lindenmayersystems}
\usetikzlibrary{decorations.pathreplacing}
\usetikzlibrary{hobby} 

\usepackage[linecolor=black!50,backgroundcolor=black!5,bordercolor=black!50,textsize=tiny]{todonotes}
\usepackage{tcolorbox}
\tcbuselibrary{breakable}

\usepackage{epigraph}
\usepackage[normalem]{ulem}

\theoremstyle{plain}
\newtheorem{thm}{Theorem}[section]

\newtheorem{lemma}[thm]{Lemma}

\newtheorem{corollary}[thm]{Corollary}
\newtheorem{prop}[thm]{Proposition}
\newtoks\prt

\theoremstyle{definition}

\newtheorem{remark}[thm]{Remark}

\newtheorem{definition}[thm]{Definition}

\def\eqn#1$$#2$${\begin{equation}\label#1#2\end{equation}}

\long\def\clrred#1\endred{{\color{red}#1}}
\long\def\clrmagenta#1\endmagenta{{\color{magenta}#1}}
\long\def\clrpurple#1\endpurple{{\color{purple}#1}}
\long\def\clrblue#1\endblue{{\color{blue}#1}}

\numberwithin{equation}{section}

\def\epsilon{\varepsilon}

\def\er{\mathbb R}
\def\ff{\varphi}
\def\rn{\er^n}
\def\cE{\mathcal E}
\def\cF{\mathcal F}

\def\cL{\mathcal L}
\def\loc{{\rm loc}}

\def\adj{\operatorname{adj}}
\def\det{\operatorname{det}}
\def\osc{\operatorname{osc}}
\def\deg{\operatorname{deg}}

\def\dist{\operatorname{dist}}
\def\dx{\,dx}
\def\dy{\,dy}
\def\INV{\mathrm{(INV)}}
\def\N{\mathrm{(N)}}

\title[Differentiability of limits of homeomorphisms]{Differentiability almost everywhere of weak limits of bi-Sobolev homeomorphisms}

\author[A. Dole\v{z}alov\'a]{Anna Dole\v{z}alov\'a}
\address{Department of Mathematical Analysis, Charles University,
So\-ko\-lovsk\'a 83, 186~00 Prague 8, Czech Republic}
\email{\tt dolezalova@karlin.mff.cuni.cz}

\author[A. Molchanova]{Anastasia Molchanova}
\address{Faculty of Mathematics, University of Vienna,
Oskar-Morgenstern-Platz 1, A-1090 Vienna, Austria}
\email{\tt anastasia.molchanova@univie.ac.at}



\thanks{
The first author was supported by Charles University Research program No. UNCE/SCI/023, by the grant GA\v{C}R P201/21-01976S and by the project Grant Schemes at CU, reg. no.
CZ.02.2.69/0.0/0.0/19 073/0016935. 
The second author was supported by the European Unions Horizon 2020 research and innovation programme under the Marie Sk\l adowska-Curie grant agreement No 847693. }

\begin{document}

\maketitle

\epigraph{\textit{\footnotesize{
[Reshetnyak's] synthesis of classical function theory and Sobolev
function classes was so fruitful that it was given a special name: quasiconformal
analysis.}}}{\footnotesize{A.\,D.\,Aleksandrov, 1999 Russ.\,Math.\,Surv.\,54 1069}
\vspace{1\baselineskip}}

\begin{abstract}
    This paper investigates the differentiability of weak limits of bi-Sobolev homeomorphisms. 
    Given $p>n-1$, consider a sequence of homeomorphisms $f_k$ with positive Jacobians $J_{f_k} >0$ almost everywhere and 
    $\sup_k(\|f_{k}\|_{W^{1,n-1}} + \|f_{k}^{-1}\|_{W^{1,p}}) <\infty$.
    We prove that if $f$ and $h$ are weak limits of $f_k$ and $f_k^{-1}$, respectively, with positive Jacobians $J_f>0$ and $J_h>0$ a.e., then $h(f(x))=x$ and $f(h(y))=y$ both hold a.e.\ and $f$ and $h$ are 
    differentiable almost everywhere.
\end{abstract}

\section{Introduction}

Let $\Omega$ and $\Omega'$ be domains, i.e.\ non-empty connected open sets, in $\rn$ and $f\in W^{1,p}(\Omega,\rn)$ be a mapping from $\Omega$ to $\Omega'$.
According to classic results of Geometric analysis, 
if $p>n$, the mapping $f$ is differentiable almost everywhere. 
This result was established in 1941 for $n=2$ by Cesari \cite{Ces1941} and later generalized to arbitrary $n$ by Calder\'{o}n \cite{Cal1951}.
The a.e.-differentiability of continuous and monotone mappings was studied from a geometrical perspective by V\"ais\"al\"a \cite{Vai1965-1} and Reshetnyak \cite{Res1966, Resh1967-2, Res1968}. 
This includes mappings with bounded distortion, also known as quasiregular mappings, and mappings with finite distortion (even for $p=n$). Further details on these results can be found in \cite{Resh1967-2, Vai1965-1}.
The results also extend to
$W^{1,1}$-homeomorphisms in dimension $n=2$, as shown by Gehring and Lehto
\cite{GehLeh1959},
and $W^{1,p}$-homeomorphisms with $p>n-1$ if $n\geq 3$, see V\"ais\"al\"a \cite{Vai1965-1} (also Onninen \cite[Theorem 1.2 and Example 1.3]{Onn2000}).

For $W^{1,n-1}$-Sobolev homeomorphisms with $n\geq 3$,
the a.e.-differentiability was established
by considering the integrability of the
inner
distortion $K_{I}\in L^{1}(\Omega)$, where 
$J_f(x):=\det Df(x)$
is the Jacobian,  $\adj Df$ is the adjugate matrix of $Df$ and
\vspace{0.1\baselineskip}
$$
    K_I:=\dfrac{|\adj Df|^n}{J_f(x)^{n-1}}, 
    \vspace{0.1\baselineskip}
$$ 
see \cite{Ten2014}.
This condition on integrability of distortion is sharp, meaning 
for any $\delta\in(0,1)$ and $n\geq 3$ there exists a homeomorphism $f\in W^{1,n-1}((-1,1)^n,\rn)$ such that
$K_I \in L^{\delta}((-1,1)^n)$ and $f$ is not classically differentiable on a set of positive measure~\cite{HenTen2017}.
The a.e.-differentiability 
of $W^{1,n-1}$-Sobolev maps also holds for continuous, open, and discrete mappings of finite distortion with nonnegative Jacobian if a particular weighted distortion function is integrable \cite{Vod2018}.
The condition $K_I\in L^{1}(\Omega)$ essentially means that $f^{-1} \in W^{1,n}(f(\Omega),\rn)$ \cite[Theorem 1.1]{MosPas2014}. 
Together with the oscillation estimate from \cite[Lemma 2.1]{Onn2006} we then obtain that for almost all $x\in \Omega$ 
\[
    \limsup_{r\to0+} \frac{\osc_{B(x,r)} f}{r}<\infty,
\]
and hence $f$ is differentiable in $x$ by the Stepanov Theorem.
Thus, instead of assumptions for distortion, we can directly consider bi-Sobolev homeomorphisms.
The inverse mapping theorem (see e.g.~\cite[Theorem A.29]{HenKos2014})
states that
if $f\in W^{1,n-1}$, $J_{f}>0$~a.e., and $f^{-1}\in W^{1,p}$ with $p>n-1$,
then both $f$ and $f^{-1}$ are differentiable almost everywhere
(for a more general approach, the reader is referred to \cite{Vod2020}).
However, Cs\"{o}rnyei, Hencl, and Mal\'{y} constructed in Example 5.2 in \cite{CsoHenMal2010}
 a homeomorphism
$f\in W^{1,n-1}((-1,1)^n,\rn)$, $n\geq 3$, with $J_{f}>0$ a.e.\ that is nowhere differentiable and its inverse $f^{-1}\in  W^{1,n-1}((-1,1)^n,\rn)$ is also nowhere differentiable.

In this work, we examine the a.e.-differentiability of a class of \textit{weak limits of homeomorphisms}. This class of mappings is well suited for the calculus of variations approach and may serve as deformations in Continuum Mechanics models. 
For further information, refer to \cite{IwaOnn2009,IwaOnnZhu2020,MolVod2020}.
Weak limits of Sobolev homeomorphisms have received significant attention in recent years, with various studies conducted, including \cite{BouHenMol2020,CamHenKauRad2018,DePhiPra2020,DolHenMal2021,DolHenMol2022,FusMosSbo2008,HenOnn2018,IwaOnn2017}.

Here 
we consider the energy functional
\[
    \cE(f):=\int_{\Omega}|D f(x)|^{n-1}\dx + \int_{\Omega'}|D f^{-1}(y)|^{p}\dy
\]
for bi-Sobolev mappings $f\colon \Omega \to \Omega'$ such that $f$ is invertible almost everywhere, $f\in W^{1,n-1}(\Omega,\mathbb{R}^n)$, and $f^{-1}\in W^{1,p}(\Omega',\mathbb{R}^n)$ for some $p>n-1$.

The main result, which is proven in Section \ref{corollary_application}, reads as follows.
\begin{thm}\label{thm:main}
    Let $n\geq 2$, $p>n-1$, $\Omega$, $\Omega'\subset\rn$ be bounded domains  
    and 
    $f_k\in W^{1,n-1}(\Omega,\rn)$, 
    $k =0,1,2\dots$, be homeomorphisms of $\overline\Omega$ onto $\overline{\Omega'}$ 
    with $J_{f_k}> 0$ a.e.\ 
    and
    \[
        \sup_k\cE(f_k)<\infty.
    \]
    Assume that $f\colon\Omega\to\rn$ is a weak limit of $\{f_k\}_{k\in \mathbb{N}}$ in $W^{1,n-1}(\Omega,\rn)$ with $J_{f}>0$ a.e.\ and $h\colon\Omega'\to\rn$ is a weak limit of $\{f_k^{-1}\}_{k\in \mathbb{N}}$ in $W^{1,p}(\Omega',\rn)$ with $J_{h}>0$ a.e.
    Then for a.e.\ $x\in\Omega$ we have $h(f(x))=x$ 
    and for a.e.\ $y\in\Omega'$ we have 
    $f(h(y))=y$,
    and both $f$ and $h$ are differentiable almost everywhere.
\end{thm}
Let us note the following result, which better suits the Calculus of Variations approach since it formulates the assumptions only for $f_k$.

\begin{corollary}\label{corollary_application}
    Let $n\geq 2$, $p>n-1$, $\Omega$, $\Omega'\subset\rn$ be bounded domains 
    and 
    $\varphi$ be a positive convex function on $(0,\infty)$ with 
    \begin{equation}\label{varphi}
        \lim_{t\to 0^+}\ff(t)=\infty\quad\text{ and }\quad \lim_{t\to \infty}\frac{\ff(t)}{t}=\infty.
    \end{equation}
    Let $f_k\in W^{1,n-1}(\Omega,\rn)$, 
    $k =0,1,2\dots$, be homeomorphisms of $\overline\Omega$ onto $\overline{\Omega'}$ with $J_{f_k}> 0$ a.e.\ 
    such that 
    $\sup_k\cF(f_k)<\infty$,
    where
    \[
        \cF(f):=\int_{\Omega}|D f(x)|^{n-1} + \frac{|\adj D f (x)|^p}{J_f^{p-1}(x)} + \varphi(J_f (x))\dx.
    \]
    Assume that $f\colon\Omega\to\rn$ is a weak limit of $\{f_k\}_{k\in \mathbb{N}}$ in $W^{1,n-1}(\Omega,\rn)$
    and $h\colon\Omega'\to\rn$ is a weak limit of $\{f_k^{-1}\}_{k\in \mathbb{N}}$ in $W^{1,p}(\Omega',\rn)$.
    Then for a.e.\ $x\in\Omega$ we have $h(f(x))=x$ 
    and for a.e.\ $y\in\Omega'$ we have 
    $f(h(y))=y$,
    and both $f$ and $h$ are differentiable almost everywhere. 
\end{corollary}


\section{Preliminaries}

By $B(c,r)$, we denote the open euclidean ball with centre $c \in \mathbb{R}^{n}$ and radius $r>0$, and $S(c, r)$ stands for the corresponding sphere.

\subsection{Topological image and $\INV$ condition}
Although a weak limit of homeomorphisms may not be a homeomorphism, it may possess an invertibility property known as the $\INV$ condition.
The $\INV$ condition states, informally, that a ball $B(x,r)$ is mapped inside the image of the sphere $f(S(x,r))$ and the complement $\Omega\setminus \overline{B(x,r)}$ is mapped outside $f(S(x,r))$.
This concept was introduced for $W^{1,p}$-mappings, where $p>n-1$, by 
M\"uller and Spector \cite{MulSpe1995}, 
although the fact that a ball $B(x,r)$ is mapped inside the image of a sphere $f(S(a,r))$  was known in literature before as \textit{monotonicity}, see \cite{Resh1967-2} and \cite[\S 2]{VodGold1976}.
Suppose that $f\colon S(y,r) \to \er^{n}$ is continuous, we define the \textit{topological image} of $B(x,r)$ as
\begin{equation}
    f^{T}(B(x,r)):=\{z\in\rn\setminus f(S(x,r)) : \deg(f,S(x,r),z)\neq 0\}
\end{equation}
and the \textit{topological image} of $x$ as
$$
f^{T}(x):= \bigcap_{r>0, r\not\in N_x} {f^*}^{T}(B(x,r))\cup f^*(S(x,r)),
$$
where $N_x$ is a null set from the definition just below.

\begin{definition}\label{def:inv}
    A mapping $f\colon\Omega\to \rn$ satisfies the \textit{$\INV$ condition},
    provided that for every $x\in \Omega$ there exist a constant $r_{x}>0$ and an $\cL^{1}$-null set $N_x$ such that for all $r\in(0,r_{x})\setminus N_{x}$, 
    the restriction $f|_{S(x,r)}$ is continuous and
    \begin{enumerate}
        \item[(i)] $f(z) \in f^{T}(B(x,r))\cup f(S(x,r))$ for a.e.~$z\in \overline{B(x,r)}$,
        \item[(ii)] $f(z) \in \rn \setminus f^{T}(B(x,r))$ for a.e.~$z\in \Omega \setminus B(x,r)$.
    \end{enumerate} 
\end{definition}

Let us note that for a particular representative of a Sobolev mapping, Definition~\ref{def:inv} allows for some points to escape their destiny, e.g.\ a null-set inside the ball may be mapped outside the image of this ball. Thus, we also consider a stronger version of the $\INV$ condition.

\begin{definition}\label{def:strong-inv}
    A mapping $f\colon\Omega\to \rn$ satisfies the \textit{strong $\INV$ condition},
    provided that for every $x\in \Omega$ there exist a constant $r_{x}>0$ and an $\cL^{1}$-null set $N_x$ such that for all $r\in(0,r_{x})\setminus N_{x}$ the restriction $f|_{S(x,r)}$ is continuous and
    \begin{enumerate}
        \item[(i)] $f(z) \in f^{T}(B(x,r))\cup f(S(x,r))$ for every $z\in \overline{B(x,r)}$,
        \item[(ii)] $f(z) \in \rn \setminus f^{T}(B(x,r))$ for every $z\in \Omega \setminus B(x,r)$.
    \end{enumerate} 
\end{definition}

\subsection{Precise, super-precise, and hyper-precise representative of a Sobolev mapping}
Let $1\leq p \leq n$ 
and $f\in W^{1,p}(\rn)$, then the \textit{precise representative} of $f$ is given by
\begin{equation}\label{def:precise}
    f^{*}(a):= 
    \begin{cases}
    \displaystyle
    \lim\limits_{r \rightarrow 0^{+}} \frac{1}{|B(a, r)|} \int_{B(a, r)}f(x) \dx & \text { if the limit exists,} \\
    0 & \text { otherwise.}
    \end{cases}
\end{equation}
Note that the representative $f^{*}$ is $p$-quasicontinuous (see remarks after \cite[Proposition~2.8]{MulSpe1995}). 

Let now $f\colon \Omega \to \rn$ be a $W^{1,p}$-weak limit of homeomorphisms $f_{k}\colon \Omega \to \rn$ with $p\in(n-1,n]$ for $n>2$ or $p \in [1,2]$ for $n=2$.
Then by \cite[Theorem 5.2]{BouHenMol2020} there exists an $\mathcal{H}^{n-p}$-null set $N C \subset \Omega$ and a representative $f^{**}$ of $f$ such that $f^{**}$ is continuous at every $x \in \Omega \backslash N C$,
a set-valued image 
$f^{T}(x)$ is a singleton 
for every $y \in \Omega \backslash N C, f^{**}=f^{*} \operatorname{cap}_{p}$-a.e., and $f^{**}$ can be chosen so that
$f^{**}(x) \in f^{T}(x)$ for every $x \in \Omega$.
We will call $f^{**}$ a \textit{super-precise representative} of $f$.

The \textit{hyper-precise} representative $\tilde{f}$ is defined as 
\begin{equation}\label{def:hyper-precise}
    \tilde{f}(a):= 
    \limsup\limits_{r \rightarrow 0^{+}} \frac{1}{|B(a, r)|} \int_{B(a, r)}f(x) \dx.
\end{equation}

We need the following monotonicity property of mappings satisfying the strong $\INV$ condition.
\begin{lemma}\label{lem:h-monotone}
    Let $n\geq 2$ and $\Omega'\subset\rn$ be a bounded domain.
    If 
    $h\colon \Omega' \to \rn$
    satisfies the strong $\INV$ condition, then $h$ is monotone for almost all radii, 
    i.e.~for $y\in\Omega'$ there exists an $\mathcal{L}^{1}$-null set $N_{y}$ such that for all $r \in (0, r_{y}) \setminus N_{y}$ it holds that
    $\osc_{B(y,r)} h \leq  \osc_{S(y,r)} h$.\\
    If, moreover,
    $h\in W^{1,p}(\Omega',\rn)$ 
    with $p>n-1$, then for any $r \in \left(0, \frac{r_{y}}{2}\right)$ the following estimate holds 
    $$
        \osc_{B(y,r)} h \leq C r \left(r^{-n}  \int_{B(y,2r)} |Dh|^p\right)^{1/p}. 
    $$
\end{lemma}

\begin{proof}
Let $N_{y}$ be a set from Definition~\ref{def:strong-inv}.
Then for $y\in\Omega'$ and $r \in (0, r_{y}) \setminus N_{y}$ it holds that 
$h$ is continuous on the sphere $S(y,r)$
and 
$h(z) \in h^{T}(B(y,r))\cup h(S(y,r))$ for every $z\in \overline{B(y,r)}$.
In this case, $h(S(y,r))$ is a compact set and $h^T(B(y,r))\subseteq \mathbb{R}^n\setminus A$, where $A$ is the unbounded component of $\mathbb{R}^n\setminus h(S(y,r))$ (since by the basic properties of the topological degree \cite[p.~48(d)]{HenKos2014} we have $\deg (h,S(y,r),\xi)=0$ for all $\xi \in A$), 
and therefore 
$\osc_{B(y,r)} h \leq  \osc_{S(y,r)} h$.

Further, for $y\in\Omega'$ and $r>0$, and for a.e.~$t \in [r, 2r)$, 
it holds that 
$$
    \osc_{B(y,r)} h \leq \osc_{B(y,t)} h \leq  \osc_{S(y,t)} h.  
$$
Then by the Sobolev embedding theorem on spheres \cite[Lemma 2.19]{HenKos2014}, following the proof of \cite[Theorem 2.24]{HenKos2014}, we obtain that 
$$
    \osc_{B(y,r)} h  \leq \osc_{S(y,t)} h \leq C t \left(t^{-n+1}  \int_{S(y,t)} |Dh|^p\right)^{1/p}\leq C r \left(r^{-n}  \int_{B(y,2r)} |Dh|^p\right)^{1/p}.     
$$
\end{proof}

\begin{remark} \label{rem:p>n}
    In case $p>n$,
    $h^*=h^{**}=\tilde{h}$
    is the continuous representative of
    $h$
    and 
    $h^{*}$ is
    differentiable almost everywhere \cite{Cal1951} and satisfies
    the Lusin $\N$ condition in $\Omega$ \cite{MarMiz1973}.
    Moreover, due to compact embedding of $W^{1,p}$ into the H\"older space $C^{0,\alpha}$, weak convergence in $W^{1,p}$ implies uniform convergence on compact sets.
    With these properties, the subsequent analysis becomes simplified, and the details are left to the reader.
\end{remark}

\section{A.e.-invertibility of $f$}

Since a limit of homeomorphisms may not be a homeomorphism, we need to define a weaker notion of inverse mapping. 
First
recall that a mapping $f\colon \Omega \to \Omega'$ is called \textit{injective a.e.~in domain} if there exists a null set $\Sigma \subset \Omega$, $|\Sigma|=0$, such that the restriction $f|_{\Omega\setminus\Sigma}\colon \Omega\setminus\Sigma \to f(\Omega\setminus\Sigma)$ is injective. 
A mapping $f\colon \Omega \to \Omega'$ is called \textit{injective a.e.~in image} if there exists a null set $\Sigma' \subset \Omega'$, $|\Sigma'|=0$, such that for any $y\in f(\Omega)\setminus \Sigma'$ the preimage $f^{-1}(y):=\{x\in\Omega : f(x)=y\}$ consists of only one point.
Note that if $f$ is injective a.e.~in image and satisfies the $\N^{-1}$ condition, then $f$ is injective a.e.~in domain. 
If instead $f$ is injective a.e.~in domain, $f$ satisfies the $\N$ condition, and $|\Omega'|=|f(\Omega)|$ 
then $f$ is injective a.e.~in image.
We say that $h\colon \Omega' \to \Omega$ is the a.e.-inverse to $f\colon \Omega \to \Omega'$ if 
for a.e.\ $x\in\Omega$ we have $h(f(x))=x$ 
and for a.e.\ $y\in\Omega'$ we have $f(h(y))=y$.
Note that if $f$ satisfies the $\N^{-1}$ condition, then $f$ is injective~a.e.~in image if and only if there exists the a.e.-inverse to $f$. 

The following lemma provides some additional conditions that guarantee the a.e.-invertibility of $f$ in our setting.

\begin{lemma}\label{lemma:injectivity}
    Let
    $n\geq 2$,
    $\Omega$ and $\Omega'$ be bounded domains in $\mathbb{R}^n$, $p>n-1$, and let $f_k\in W^{1,n-1}(\Omega,\rn)$ be homeomorphisms of $\overline\Omega$ onto $\overline{\Omega'}$ with $J_{f_k}>0$. Let also $f\colon\Omega\to\rn$ be a weak limit of $\{f_k\}_{k\in\mathbb{N}}$ in $W^{1,n-1}(\Omega,\rn)$ with $J_f>0$~a.e. 
    Assume also that 
    the sequence 
    $\{f_k^{-1}\}_{k\in\mathbb{N}}$ 
    converges $W^{1,p}$-weakly to $h\colon \Omega' \to \mathbb{R}^{n}$ with $J_h>0$ a.e.
    Then $h^{**}(f(x))=x$ a.e.~in $\Omega$ and $f(h^{**}(y))=y$ a.e.~in $\Omega'$.
\end{lemma}
\begin{proof}
Let $p>n-1$, and fix a representative of $f$, which we denote by the same symbol.
If needed, we pass to a subsequence so that $f_k\to f$ and $f_k^{-1}\to h$ pointwise a.e.
Since $h$ is a $W^{1,p}$-weak limit of Sobolev homeomorphisms with $p>n-1$, 
the super-precise representative $h^{**}$ satisfies the strong $\INV$ condition \cite[Theorem 5.2 and Lemma 5.3]{BouHenMol2020}. 
Then there exists a set $G_1'\subset\Omega'$ of full measure $|G_1'|=|\Omega'|$: $J_{h^{**}}(y) >0$ for all $y\in G_1'$, $h^{**}$ is injective in $G_1'$ (see \cite[Lemma 3.4]{MulSpe1995} and \cite[Theorem 1.2]{BouHenMol2020})
and $f^{-1}_{k}(y) \to h^{**}(y)$ for all $y\in G_1'$. 

{\underline{Step 1. $h^{**}(f(x))=x$ a.e.}:}
By Lemma~\ref{lem:h-monotone}, we know that 
$\osc_{B(y,r)}h^{**}\underset{r\to 0}{\longrightarrow}0$ for a.e.~$y\in\Omega'$. Since $J_{f}>0$ a.e.~(and therefore $f$ satisfies the $\N^{-1}$ condition), \linebreak
$\osc_{B(f(x),r)}h^{**}\underset{r\to 0}{\longrightarrow} 0$ for a.e.~$x\in\Omega$.

Let $G_{1} \subset f^{-1}(G_1')$ be a set such that $|G_{1}|=|\Omega|$ and for all $x\in G_1$ it holds that $f_k(x)\to f(x)$ and $\osc_{B(f(x),r)}h^{**}\underset{r\to 0}{\longrightarrow} 0$.

For $x\in G_1$ and $r>0$, by the pointwise convergence of $f_{k}$ in $x\in G_{1}$ and $f^{-1}_{k}$ in $f(x)\in G_{1}'$, we can find $k_0\in\mathbb{N}$ big enough such that 
$$
    f_k(x)\in B(f(x), r) \quad \text{ and } \quad f_k^{-1}(f(x))\in B(h^{**}(f(x)), r)
$$
for all $k\geq k_0$.
Moreover, by \cite[Lemma 2.9]{MulSpe1995} (though it is formulated for the precise representative $h^{*}$, it holds also for the super-precise representative $h^{**}$ with analogous proof), 
there exists a subsequence $\{f_{k_j}\}_{j \in \mathbb{N}}$ (that depends on $r$) and a number $j_{0} \in \mathbb{N}$ big enough such that
$$
    \osc_{S(f(x), r)} f_{k_j}^{-1}\leq \osc_{S(f(x), r)} h^{**} + r
$$ 
for all $j\geq j_0$.

Then we have
\begin{align*}
    |f_{k_j}^{-1}(f_{k_j}(x)) - h^{**}(f(x))| &\leq |f_{k_j}^{-1}(f_{k_j}(x)) - f_{k_j}^{-1}(f(x))|  + |f_{k_j}^{-1}(f(x)) - h^{**}(f(x))| \\
    & \leq  \osc_{B(f(x), r)} f_{k_j}^{-1} + r \leq \osc_{S(f(x), r)} f_{k_j}^{-1} + r \\
    & \leq \osc_{S(f(x), r)} h^{**} +r +r   \leq \osc_{B(f(x), 2r)} h^{**} + 2r.
\end{align*}
Therefore, by definition of $G_1$,
$$
    |x - h^{**}(f(x))| = |f_{k_j}^{-1}(f_{k_j}(x)) - h^{**}(f(x))| \leq \lim\limits_{r\to 0} (\osc_{B(f(x), 2r)} h^{**} + 2r) =0
$$
for all $x\in G_1$, which concludes Step 1.

{\underline{Step 2.  $f(h^{**}(y))=y$ a.e.}:}
We know that $h^{**}$ is injective a.e.\ on $G_1'$ and both $f$ and $h^{**}$ satisfies the $\N^{-1}$ condition, so when we set 
$$
G_2':=\left(G'_1\cap (h^{**})^{-1}(G_1)\right)\setminus (h^{**})^{-1}(f^{-1}(\Omega'\setminus G_1')),
$$ 
we know it is a set of full measure. Let us take $y\in G_2'$. 
Since $f_k^{-1}$ is a homeomorphism onto $\Omega$, we can find $y_k\in\Omega'$ such that $f_k^{-1}(y_k)=h^{**}(y)$. Therefore,
$$
y_k=f_k(f_k^{-1}(y_k))=f_k(h^{**}(y))\to f(h^{**}(y)),
$$
so $y_k$ converges to some $\tilde{y}=f(h^{**}(y))$. We apply $h^{**}$ to both sides to get $h^{**}(\tilde{y}) = h^{**}(f(h^{**}(y)))$. From $y\in G_2'$ we have that $h^{**}(y)\in G_1$. Since $h^{**}(f(x))=x$ on $G_1$ we get 
$
h^{**}(\tilde{y})= h^{**}(f(h^{**}(y)))=h^{**}(y).
$
Now we can have either $\tilde{y} \in G_1'$ or $\tilde{y}\notin G_1'$. In the first case, $\tilde{y}=y$ as $h^{**}$ is injective on $G_1'$, so $f(h^{**}(y))=y$. 
In the other case, $f(h^{**}(y))\in \Omega'\setminus G_1'$, which is a contradiction to $y\in G_2'$.

\end{proof}
\begin{remark}
If $p>n$, equality $h^{**}(f(x))=x$ can be derived easily from
$$
    |x - h^{**}(f(x))| \leq 
    |f_{k}^{-1}(f_{k}(x)) - f_{k}^{-1}(f(x))| + |f_{k}^{-1}(f(x)) - h^{**}(f(x))|,
$$
using  uniform convergence $f_k^{-1}\rightrightarrows h^{**}$ (up to subsequence) and the Morrey inequality for $f_k^{-1}$.
The other relation $f(h^{**}(y))=y$ follows the same way as above. 

\end{remark}
\begin{remark}
Since both $f$ and $h$ satisfy the $\N^{-1}$ condition, the identities $h(f(x))=x$ a.e.~in $\Omega$ and $f(h(y))=y$ a.e.~in $\Omega'$ hold for arbitrary representatives.
\end{remark}


\section{Differentiability}

First, let us notice the following well-known fact. 

\begin{lemma}\label{lem:h-diff}
    Let $n\geq 2$, $p>n-1$ and $\Omega'$ be a bounded domain in $\mathbb{R}^n$.
    If $h\in W^{1,p}_{\loc}(\Omega',\rn)$
    satisfies the strong $\INV$ condition, 
    then $h$ is differentiable a.e.~in $\Omega'$.
\end{lemma}
\begin{proof}
By Lemma \ref{lem:h-monotone} we have
$$
    \osc_{B(y,r)} h \leq C r \left(r^{-n}  \int_{B(y,2r)} |Dh|^p\right)^{1/p}, 
$$

which implies by setting $r=|z-y|$ that  
$$
\limsup_{z\to y}\frac{|h(z)-h(y)|}{|z-y|}\leq C |Dh(y)| <\infty
$$
for any Lebesgue point $y$ of $|Dh|^p$
and, therefore, $h$ is differentiable a.e.~by the Stepanov theorem \cite{Ste1923}, see also \cite[Theorem 2.23]{HenKos2014}.

\end{proof}

We also need the following modification of \cite[Lemma A.29]{HenKos2014}, which gives us the 
a.e.-differentiability  
of mapping $f$ from Theorem \ref{thm:main} -- but the derivative is only with respect to a set of full measure.
\begin{lemma}\label{lem:A29}
    Let $n\geq 2$ and $\Omega$, $\Omega'$ be bounded domains in $\mathbb{R}^n$.   
    Let $\Lambda\subset\Omega$, $\Lambda'\subset\Omega'$ be sets of full measure and $h\colon\Omega'\to\Omega$ such that $h\colon\Lambda'\to \Lambda=h(\Lambda')$ is differentiable with respect to the relative topology in $\Lambda'$, i.e.~induced by the topology in $\rn$, 
    and $J_{h}(y) > 0$ for all $y\in \Lambda'$.
    Assume also that $h|_{\Lambda'}$ 
    is injective, and the inverse mapping $f:=h^{-1}$ is continuous in $\Lambda$ with respect to the relative topology in $\Lambda$.
    Then $f$ is differentiable on $\Lambda$ with respect to the relative topology in $\Lambda$ and $Df(x)=\left(Dh(f(x))\right)^{-1}$ for all $x\in\Lambda$.
\end{lemma}
\begin{proof} 
Since $h\colon \Lambda' \to \Lambda$ is a homeomorphism,
the proof of this lemma follows the lines of the proof of \cite[Lemma A.29]{HenKos2014}. We present it here for the convenience of the reader.

By the differentiability of $h$ we know that for $y\in\Lambda'$
\begin{equation}\label{eq:diff}
    \lim\limits_{\bar{y}\to y, \, \bar{y}\in\Lambda'} \frac{h(\bar{y})-h(y)-Dh(y)(\bar{y}-y)}{|\bar{y}-y|}=0.
\end{equation}

For $\bar{x}$, $x \in \Lambda$ denote $\bar{y}=f(\bar{x})$, $y=f(x) \in \Lambda'$, then
$$
    h(\bar{y})-h(y)=h(f(\bar{x}))-h(f(x)) = \bar{x}-x.
$$
Since $J_h(y) >0$ we obtain for $\bar{y}$ close enough to $y$  that
$$
|\bar{x}-x|=|h(\bar{y})-h(y)|\approx |Dh(y)(\bar{y}-y)|\approx|\bar{y}-y|.
$$

Then from \eqref{eq:diff} it follows
\begin{align*}
    0 =& \lim_{\bar{y}\to y, \, \bar{y}\in\Lambda'}\frac{\left(Dh(y)\right)^{-1}\left( h(\bar{y})-h(y)-Dh(y)(\bar{y}-y)\right)}{|y'-y|} = \\
    & \lim_{{\bar{y}}\to y, \, \bar{y}\in\Lambda'}\frac{\left(Dh(y)\right)^{-1}\left( h(\bar{y})-h(y)\right)-(\bar{y}-y)}{|y'-y|} \approx \\
    & \lim_{\bar{x}\to x, \, \bar{x}\in\Lambda}\frac{\left(Dh(f(x))\right)^{-1}\left(\bar{x}-x\right)-(f(\bar{x})-f(x))}{|\bar{x}-x|},
\end{align*}
which concludes the proof.

\end{proof}

The following proposition is a version of an inverse function theorem.
\begin{prop}\label{prop:h-ball}
    Let $n\geq 2$, $p>n-1$, $\Omega$ and $\Omega'$ be bounded domains in $\mathbb{R}^n$,
    $\Lambda \subset \Omega$ and $\Lambda' \subset \Omega'$ be sets of full measure and $h\in W^{1,p} (\Omega', \Omega)$ satisfy the strong $\INV$ condition and be differentiable with $J_{h}(y)>0$ for any $y\in\Lambda'$.
    Assume also that the restriction $h|_{\Lambda'}\colon \Lambda' \to \Lambda$ is one-to-one 
    for any $y\in\Lambda'$.
    Then for any $y_{0}\in \Lambda'$ there exists a sequence $\{r_m\}_{m\in \mathbb{N}}\searrow 0$ such that the topological image $h^T(B(y_0,r_m))$ contains $B\left(h(y_0),\frac{r_m}{3}\right)$.
\end{prop}

\vspace{-4\baselineskip}
\begin{figure}[h t p]
\hspace{-15pt}
\begin{minipage}{1\textwidth}
\begin{tikzpicture}[scale=0.8]
\draw[red] (-3,3.5) circle (1.2);
\path[draw,use Hobby shortcut,closed=true]
(-2.5,0) .. (0,1) .. (-0.5,3) .. (-0.2,4) .. (-3.5,2) .. (-1.5,.5);
\node at (-1.1,0.5) {$\Omega'$};
\node[red] at (-1.3,4.6) {$S(0,r_m)$};

\draw[->] (1,2)--(3,2);
\node at (2,2.3) {$h$};

\path[draw,use Hobby shortcut,closed=true]
(4.5,-1) .. (7,1) .. (5.5,6) .. (5,5) .. (3.5,4) .. (4.5,0.5);

\draw[red] (5.5,3.5) circle (1.2);
\draw[dotted] (5.5,3.5) circle (0.6);
\draw[gray, fill=gray] (5.5,3.5) circle (0.4);
\path[blue,draw,use Hobby shortcut,closed=true]
(5.5,2.6) .. (6.2,4) .. (5.5,4.4) .. (5,4.8) .. (4.5,4) .. (4.5,2.5);
\node[red] at (5.9,1.9) {$S(0,r_m)$};
\node[blue] at (3.6,5.1) {$h(S(0,r_m))$};

\node at (4.9,0.5) {$\Omega$};

\end{tikzpicture}
\end{minipage}
\caption{Mapping $h$ maps the red sphere $S(0,r_m)$ to $h(S(0,r_m))$ (blue); 
the grey ball $B(0,r_m/3)$ does not intersect $h(S(0,r_m))$, since 
its distance from $0$ 
is at least $r_m/2$ (denoted by the dotted sphere).}\label{fig_degree}
\end{figure}
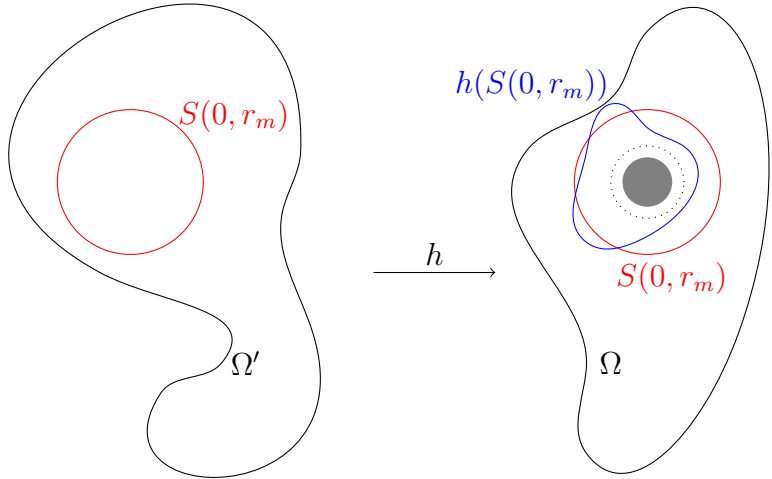

\begin{proof}
Without loss of generality, by a translation and a linear change of variables, we may assume that 
$y_0=0$, $h(y_0)=0$, and $Dh(y_0)=Id$. 
Since $h$ is differentiable at $0$, it holds that 
$h(y) = y + o(|y|)$ if $y\to 0$.
That means that there exists  $r_{0}>0$ such that 
\begin{equation}\label{neq:h_diff}
    |h(y)-y|\leq\frac{|y|}{2} \quad \text{for all } y\in B(0,r_0)\subset\Omega'. 
\end{equation}
Consider a sequence $\{r_m\}_{m\in \mathbb{N}}\searrow 0$ such that $h$ is continuous on $S(0,r_m)$
and Definition~\ref{def:strong-inv} (i--ii) is fulfilled.
Let now $z\in B\left(0,\frac{r_m}{3}\right)\subset\Omega$,
the inequality \eqref{neq:h_diff} implies
$z\notin h(S(0,r_m))$. 
Since $\dist(z,S(0,r_m))> r_m/2$, from \eqref{neq:h_diff} we know that $1=\deg(z,Id, S(0,r_m))=\deg(z,h, S(0,r_m))$. Therefore, $B\left(0, \frac{r_m}{3}\right) \subset h^T(B(0,r_m))$, see Figure \ref{fig_degree} for illustration.
\end{proof}

The closing theorem of this section concludes the differentiability part of Theorem~\ref{thm:main}.

\begin{thm} \label{thm:diff}
Let $n\geq 2$, $p>n-1$, $\Omega$ and $\Omega'$ be bounded domains in $\mathbb{R}^n$ and 
$f_k\in W^{1,n-1}(\Omega, \rn)$ be homeomorphisms of $\overline\Omega$ onto $\overline{\Omega'}$  with $J_{f_k}>0$. Let  $f\colon\Omega\to\rn$ be a weak limit of $\{f_k\}_{k\in\mathbb{N}}$ in $W^{1,n-1}(\Omega,\rn)$ with $J_f>0$~a.e.
Assume also that 
the sequence 
$\{f_k^{-1}\}_{k\in\mathbb{N}}$ 
converges $W^{1,p}$-weakly to $h\colon \Omega' \to \mathbb{R}^{n}$ with $J_h>0$ a.e.
Then $h^{**}$ is differentiable a.e.~in $\Omega'$ and $\tilde{f}$ is differentiable a.e.~in $\Omega$.
\end{thm}

\begin{proof}
We again pass to a subsequence (if needed) so that $f_k\to f$ and $f_k^{-1}\to h$ pointwise a.e. Since $h$ is a $W^{1,p}$-weak limit of Sobolev homeomorphisms with $p>n-1$, the super-precise representative $h^{**}$ satisfies 
the strong $\INV$ condition \cite[Theorem 5.2 and Lemma 5.3]{BouHenMol2020}, 
is injective a.e.\ (see \cite[Lemma 3.4]{MulSpe1995} and \cite[Theorem 1.2]{BouHenMol2020})
and continuous on almost all spheres \cite[Lemma 2.19]{HenKos2005}. 
By Lemma~\ref{lem:h-diff}, $h$ is differentiable a.e.~in $\Omega'$.
Moreover, since $J_{h}(y) > 0$ a.e.~in $\Omega'$, by the change-of-variable formula we conclude that $h$ satisfies the $\N^{-1}$ condition.

{\underline{Step 1. Finding sets $\Lambda$, $\Lambda'$}:} 
Let $f$ be an arbitrarily fixed representative, and let us
introduce \textit{good} sets $G\subset\Omega$, $G'\subset \Omega'$ as
\begin{equation*}
    G := \{x\in\Omega\: : \: h^{**}({f}(x))=x\} \subset \Omega \qquad \text{and} \qquad
    G':= \{y\in\Omega'\: : \: f(h^{**}(y))=y\} \subset \Omega'.
\end{equation*} 
It is easy to check
that
$f(G)=G'$, $h^{**}(G')=G$, and by Lemma~\ref{lemma:injectivity}, $|G|=|\Omega|$, $|G'|=|\Omega'|$.
And we define \textit{bad} sets $\Sigma \subset G$, $\Sigma'\subset G'$ as
\begin{align*}
    \Sigma:= & G \setminus \{x\in\Omega\: : \: J_f(x)>0, \: f_k(x)\to f(x)\},\\
    \Sigma':= & G' \setminus \{y\in\Omega'\: : \: 
    \: h^{**} \text{ is differentiable in } y, \: J_{h^{**}}(y)>0, \: f_{k}^{-1}(y) \to h^{**}(y)\}.
\end{align*} 
Clearly $|\Sigma|=|\Sigma'|=0$.
Then \textit{very good} sets $\Lambda\subset G$, $\Lambda'\subset G'$ are defined by
\begin{equation*}
    \Lambda':= G' \setminus (\Sigma' \cup f^{-1}(\Sigma)) \qquad \text{and} \qquad
    \Lambda := h^{**}(\Lambda').
\end{equation*} 
By Lemma~\ref{lemma:injectivity} and $\N^{-1}$ condition for $f$ and $h^{**}$, it is not difficult to see that
$|\Lambda'|=|G'|=|\Omega'|$,  $|\Lambda|=|\Omega|$ and $ f(\Lambda)=\Lambda'$.

{\underline{Step 2. $f|_{\Lambda}$ is continuous}:} 
The restriction $f|_{\Lambda}\colon\Lambda \to \Lambda'$ is continuous with respect to the relative topology in $\Lambda$.
Indeed,
let $f|_{\Lambda}$ be not continuous in some point $x_{0}\in\Lambda$, then there exists a sequence
$\{x_{k}\}_{k\in\mathbb N}\subset \Lambda$,
$x_{k} \to x_{0}$, but 
$f(x_{k})\nrightarrow f(x_{0})$. 
We set $y_{k}:=f(x_{k})\in \Lambda'$ and $y_0:= (h^{**})^{-1}(x_0)=f(x_0)$. 
Since $h^{**}|_{\Lambda'}=(f|_{\Lambda})^{-1}$, we have 
$h^{**}(y_{k}) \to h^{**}(y_{0})$,
but
$y_{k} \nrightarrow y_{0}$. 

By Proposition~\ref{prop:h-ball} there exists a sequence
$\{r_m\}_{m\in\mathbb{N}}\searrow 0$ such that 
$$
B\left(h^{**}(y_0),\frac{r_m}{3}\right) \subset (h^{**})^T(B(y_0,r_m)).
$$
Let $m$ and $k_0\in \mathbb{N}$ be big enough so that infinitely many $y_k$ are outside of $B(y_0,r_m)$ for $k\geq k_0$
and $h^{**}(y_{k_0})\in B\left(h^{**}(y_0), \frac{r_m}{6}\right)$. 
Passing to a subsequence, we can, for now, assume that $y_k\notin B(y_0, r_m)$ for all $k$.
Then we can find $r>0$ such that 
$$B(y_{k_0}, r)\cap B(y_0,r_m)=\varnothing
$$ 
and, since $h^{**}|_{\Lambda'}$ is continuous,
$$
h^{**}(B(y_{k_0},r)\cap \Lambda')\subset B\left(h^{**}(y_{k_0}), \frac{r_m}{6}\right).
$$
Summarizing the above, we obtain  
$$
h^{**}(B(y_{k_0},r)\cap \Lambda')\subset B\left(h^{**}(y_{k_0}),\frac{r_m}{6}\right)\subset B\left(h^{**}(y_0), \frac{r_m}{3}\right)\subset (h^{**})^T(B(y_0,r_m)).
$$
Thus, for every
$$z\in (B(y_{k_0}, r)\cap\Lambda' )\subset (\Omega' \setminus B(y_0,r_m))$$ it holds that
$h^{**}(z) \in (h^{**})^{T} (B(y_0,r_m))$,
the latter contradicts to the strong $\INV$ condition for $h^{**}$, since a set of positive measure $B(y_{k_0}, r)\cap\Lambda'$ from outside of the ball $B(y_0,r_m)$ is mapped inside the topological image of this ball.

Therefore, $f$ is continuous on $\Lambda$ with respect to the relative topology, and by Lemma~\ref{lem:A29}, we conclude that $f$ is differentiable on $\Lambda$ with respect to the relative topology.

\medskip
{\underline{Step 3. $\tilde{f}$ is differentiable a.e.}:}
It is left to show that a hyper-precise representative $\tilde{f}$, given by \eqref{def:hyper-precise}, is differentiable at $x_0\in \Lambda$ with respect to $\Omega$.
Since $\Lambda$ is a set of full measure and $f$ is continuous on $\Lambda$ with respect to the relative topology, any point $x\in \Lambda$ is a Lebesgue point of $f$, and therefore $\tilde{f}=f$ on $\Lambda$.

Fix $x_0\in\Lambda$ and $\varepsilon>0$. 
By differentiability of $f$ on $\Lambda$ with respect to the relative topology,
there exists $s>0$ such that for any $x\in B(x_0,s)\cap\Lambda$ it holds that
\begin{equation}\label{neq:diff_in_lambda}
\frac{|f(x)-f(x_0) -  Df(x_0)(x-x_0)|}{|x-x_0|}=\frac{|\tilde{f}(x)-\tilde{f}(x_0) -  Df(x_0)(x-x_0)|}{|x-x_0|}<\frac{\varepsilon}{2},
\end{equation}
where $Df(x_0)$ denotes the derivative $Df|_{\Lambda}(x_0)$ with respect to the relative topology.
To prove differentiability of $\tilde{f}$, we need to show that 
for an arbitrary $x'$ close to $x_{0}$ it holds that 
\begin{equation}\label{neq:diff_in_omega}
\frac{|\tilde{f}(x')-\tilde{f}(x_0) -  Df(x_0)(x'-x_0)|}{|x'-x_0|}<\varepsilon.
\end{equation}
If $x' \in \Lambda$, \eqref{neq:diff_in_omega} follows immediately from~\eqref{neq:diff_in_lambda}.
In the other case, roughly speaking, we want to find a point $z \in \Lambda$ 
such that 
$\frac{|\tilde{f}(x')- \tilde{f}(z)|}{|x'-x_0|}$
and 
$\frac{|x'- z|}{|x'-x_0|}$ are small,
and
so we can estimate
\begin{equation*}\label{neq:stepanov0}
\begin{aligned}
    &\frac{|\tilde{f}(x')-\tilde{f}(x_0)-Df(x_0)(x'-x_0)|}{|x'-x_0|}\\
    &\qquad \leq \frac{|\tilde{f}(x')- \tilde{f}(z)|+|Df(x_0)(x'-z) |}{|x'-x_0|}+\frac{| \tilde{f}(z)-\tilde{f}(x_0) -  Df(x_0)(z-x_0)|}{|x'-x_0|}
    < \varepsilon.\\
\end{aligned}
\end{equation*} 

Now we prove the above paragraph rigorously.
Let $x'\in B\left(x_0,\frac{s}{2}\right)$.
By \eqref{def:hyper-precise}, there exists a sequence 
$\{r_k\}_{k\in\mathbb{N}}\searrow 0$ such that 
$r_k<2^{-k}|x'-x_0|$ 
and
\begin{equation}\label{neq:Lebesgue_point}
    \left|\tilde{f}(x')-\frac{1}{|B(x',r_k)|}\int_{B(x',r_k)\cap\Lambda}  \tilde{f}(x) \dx \right|<2^{-k}|x'-x_0|.
\end{equation}
In the following, we proceed coordination-wise for $i\in\{1,\dots,n\}$.
Denote by $a_k^i$ and $b_k^i$ points in $B(x',r_k)\cap\Lambda$ such that
\begin{align}
     \tilde{f}_i(a_k^i)& \geq \frac{1}{|B(x',r_k)|} \int_{B(x',r_k)\cap\Lambda} f_i(x)\dx-2^{-k}|x'-x_0|, 
    \label{def:a_k}\\
     \tilde{f}_i(b_k^i)& \leq \frac{1}{|B(x',r_k)|}\int_{B(x',r_k)\cap\Lambda} f_i(x)\dx+2^{-k}|x'-x_0|.
    \label{def:b_k}
\end{align}
If there is an equality in \eqref{def:a_k} or \eqref{def:b_k}, we define
$x_k^i$ as $a_k^i$ or $b_k^i$, correspondingly. 
Otherwise, by continuity of $ \tilde{f}_i$ on $\Lambda$, there exist two balls 
$B(a_k^i, \rho(a_k^i))$ and $B(b_k^i, \rho(b_k^i))$, contained in $B(x',r_k)$, such that 
\eqref{def:a_k} holds for any $a \in B(a_k^i, \rho(a_k^i)) \cap \Lambda$
and 
\eqref{def:b_k} holds for any $b \in B(b_k^i, \rho(b_k^i)) \cap \Lambda$.
Without loss of generality, we may assume
$a_k^i=(0,\dots,0)$ and $b_k^i=(b_1,0,\dots,0)$. 
Let us now consider the lines $l_d:=(t,d_2,\dots,d_n)$ connecting $B(a_k^i, \rho(a_k^i))$ and $B(b_k^i, \rho(b_k^i))$.
Since $\Lambda$ is of full measure, for $\mathcal{L}^{n-1}$-a.e.~$d:=(d_2,\dots,d_n)$ a line $l_{d}$ contains $x_{a} \in B(a_k^i, \rho(a_k^i)) \cap \Lambda$ and $x_{b} \in B(b_k^i, \rho(b_k^i)) \cap \Lambda$, and
$\mathcal{L}^1(l_d \setminus \Lambda)=0$.
Moreover, $ \tilde{f}_i\in W^{1,n-1}$ and hence $ \tilde{f}_i$ is absolutely continuous on $\mathcal{L}^{n-1}$-a.e.\ $l_{d}$.
Therefore, by the intermediate value property, there is a point $c_k^i\in l_{d}$ such that
\begin{equation}\label{eq:c_k}
    \left| \tilde{f}_i(c_k^i)-\frac{1}{|B(x',r_k)|}\int_{B(x',r_k)\cap\Lambda}  \tilde{f}_i(x) \dx \right|\leq 2^{-k}|x'-x_0|.
\end{equation}
Moreover, there exists $x_k^i\in l_d\cap \Lambda \subset B(x',r_k)$ such that
\begin{equation}\label{neq:fc_k-fx_k}
    | \tilde{f}(c_k^i) -  \tilde{f}(x_k^i)|\leq 2^{-k}|x'-x_0|.
\end{equation}
Then, by \eqref{neq:Lebesgue_point}, \eqref{eq:c_k}, and \eqref{neq:fc_k-fx_k},
\begin{equation}\label{neq:fx_k-fx'}
    | \tilde{f}_i(x_k^i)-\tilde{f}_i(x')|\leq | \tilde{f}_i(x_k^i)- \tilde{f}_i(c_k^i)|+| \tilde{f}_i(c_k^i)-\tilde{f}_i(x')|<2^{-k+2}|x'-x_0|.
\end{equation}

Further,
\begin{equation}\label{neq:stepanov1}
\begin{aligned}
    &\frac{|\tilde{f}_i(x')-\tilde{f}_i(x_0)-Df_i(x_0)(x'-x_0)|}{|x'-x_0|}\\
    &\qquad \leq \frac{|\tilde{f}_i(x')- \tilde{f}_i(x_k^i)|+|Df_i(x_0)(x'-x_k^i) |}{|x'-x_0|}+\frac{| \tilde{f}_i(x_k^i)-\tilde{f}_i(x_0) -  Df_i(x_0)(x_k^i-x_0)|}{|x'-x_0|}.\\
\end{aligned}
\end{equation}
Since $x_k^i\in B(x',r_k)$ and \eqref{neq:fx_k-fx'} holds, the first term in \eqref{neq:stepanov1} can be estimated as
\begin{equation*}
    \frac{|\tilde{f}_i(x')- \tilde{f}_i(x_k^i)|+|Df_i(x_0)(x'-x_k) |}{|x'-x_0|} \leq 
    2^{-k+2}+ 2^{-k}|Df(x_0)|. 
\end{equation*}
While to estimate the second term in \eqref{neq:stepanov1}, we note that 
\begin{equation*}
    |x_k^i-x_0| \leq |x_k^i-x'|+|x'-x_0|\leq (1+2^{-k})|x'-x_0| \leq 2 |x'-x_0| \leq s,
\end{equation*}
since $x_k^i\in B(x',r_k)$.
And hence, by \eqref{neq:diff_in_lambda}, we conclude
\begin{equation*}
    \frac{| \tilde{f}_i(x_k^i)-\tilde{f}_i(x_0) -  Df_i(x_0)(x_k^i-x_0)|}{|x'-x_0|} 
    \leq \frac{2| \tilde{f}_i(x_k^i)-\tilde{f}_i(x_0) -  Df_i(x_0)(x_k^i-x_0)|}{|x_k^i-x_0|}
    \leq \varepsilon.
\end{equation*}

Summarizing the above, we obtain that for $x_{0}\in \Lambda$ and any $\varepsilon >0$ there exists $s>0$ such that for any $x'\in B\left(x_0, \frac{s}{2}\right)$ it holds
\begin{equation*}
    \frac{|\tilde{f}_i(x')-\tilde{f}_i(x_0)-Df_i(x_0)(x'-x_0)|}{|x'-x_0|}\leq  \liminf \limits_{k\to \infty}(2^{-k} (4+ |Df_i(x_0)|)+\varepsilon) = \varepsilon.
\end{equation*}
Therefore, $\tilde{f}_i$ is differentiable in any $x_{0}\in\Lambda$ with respect to $\Omega$ and, moreover, 
$D\tilde{f}_i(x_0) = Df_i|_{\Lambda}(x_0)$.

\end{proof}

\vspace{-1.5\baselineskip}
\section{Proofs of Theorem \ref{thm:main} and Corollary \ref{corollary_application}}

\begin{proof}[Proof of Theorem \ref{thm:main}]
Theorem \ref{thm:main} immediately follows from Lemma~\ref{lemma:injectivity} and Theorem~\ref{thm:diff}.
\end{proof}

\begin{proof}[Proof of Corollary \ref{corollary_application}]
    Let us first note that following the proof of \cite[Theorem 1.1]{MosPas2014} with substituting $n$ by $p$, we obtain 
    \[
        \int_{\Omega'} |D f_k^{-1}|^p (y)\,dy\leq\int_\Omega \frac{|\adj Df_k|^p (x)}{(J_{f_k}(x))^{p-1}}\,dx.
    \]
    Hence, $\cE(f_k) \leq \cF(f_k)$ and the sequence $\{f_k^{-1}\}_{k\in\mathbb{N}}$ is bounded in $W^{1,p}(\Omega',\rn)$ and passing to a subsequence if needed, there exists a weak limit $h$. 
    Moreover, by \cite[Lemma 2.3]{DolHenMol2022} and~\eqref{varphi}, the inequality
    \[
        \int_{\Omega}\varphi(J_f (x))\dx \leq C
    \]
guarantees that $J_{f} > 0$ a.e.\ in $\Omega$
and $J_{h} > 0$ a.e.\ in $\Omega'$. 
To finish the proof, we apply Theorem~\ref{thm:main}.
\end{proof}

\section*{Acknowledgement}
We want to express their appreciation to Professor Stanislav Hencl for attracting our attention to the problem and for many fruitful discussions.

\bibliographystyle{plain}
\bibliography{biblio_inv}

\begin{thebibliography}{10}

\bibitem{BouHenMol2020}
O.~Bouchala, S.~Hencl, and A.~Molchanova.
\newblock Injectivity almost everywhere for weak limits of {S}obolev
  homeomorphisms.
\newblock {\em J. Funct. Anal.}, 279(7):108658, 32, 2020.

\bibitem{Cal1951}
A.~P. Calder\'{o}n.
\newblock On the differentiability of absolutely continuous functions.
\newblock {\em Riv. Mat. Univ. Parma}, 2:203--213, 1951.

\bibitem{CamHenKauRad2018}
D.~Campbell, S.~Hencl, A.~Kauranen, and E.~Radici.
\newblock Strict limits of planar {BV} homeomorphisms.
\newblock {\em Nonlinear Anal.}, 177(part A):209--237, 2018.

\bibitem{Ces1941}
L.~Cesari.
\newblock Sulle funzioni assolutamente continue in due variabili.
\newblock {\em Ann. Scuola Norm. Super. Pisa Cl. Sci. (2)}, 10:91--101, 1941.

\bibitem{CsoHenMal2010}
M.~Cs\"{o}rnyei, S.~Hencl, and J.~Mal\'{y}.
\newblock Homeomorphisms in the {S}obolev space {$W^{1,n-1}$}.
\newblock {\em J. Reine Angew. Math.}, 644:221--235, 2010.

\bibitem{DePhiPra2020}
G.~De~Philippis and A.~Pratelli.
\newblock The closure of planar diffeomorphisms in {S}obolev spaces.
\newblock {\em Ann. Inst. H. Poincar\'{e} C Anal. Non Lin\'{e}aire},
  37(1):181--224, 2020.

\bibitem{DolHenMal2021}
A.~Dole\v{z}alov\'a, S.~Hencl, and J.~Mal\'y.
\newblock Weak limit of homeomorphisms in {$W^{1,n-1}$} and {(INV)} condition.
\newblock {\em arXiv}, 2021.

\bibitem{DolHenMol2022}
A.~Doležalová, S.~Hencl, and A.~Molchanova.
\newblock Weak limit of homeomorphisms in {$W^{1,n-1}$}: invertibility and
  lower semicontinuity of energy.
\newblock {\em arXiv}, 2022.

\bibitem{FusMosSbo2008}
N.~Fusco, G.~Moscariello, and C.~Sbordone.
\newblock The limit of {$W^{1,1}$} homeomorphisms with finite distortion.
\newblock {\em Calc. Var. Partial Differential Equations}, 33(3):377--390,
  2008.

\bibitem{GehLeh1959}
F.~W. Gehring and O.~Lehto.
\newblock On the total differentiability of functions of a complex variable.
\newblock {\em Ann. Acad. Sci. Fenn. Ser. A I No.}, 272:9, 1959.

\bibitem{HenKos2005}
S.~Hencl and P.~Koskela.
\newblock Mappings of finite distortion: discreteness and openness for
  quasilight mappings.
\newblock {\em Ann. Inst. H. Poincar\'{e} Anal. Non Lin\'{e}aire},
  22(3):331--342, 2005.

\bibitem{HenKos2014}
S.~Hencl and P.~Koskela.
\newblock {\em Lectures on mappings of finite distortion}.
\newblock Lecture Notes in Mathematics, Vol. 2096. Springer International
  Publishing, 2014.

\bibitem{HenOnn2018}
S.~Hencl and J.~Onninen.
\newblock Jacobian of weak limits of {S}obolev homeomorphisms.
\newblock {\em Adv. Calc. Var.}, 11(1):65--73, 2018.

\bibitem{HenTen2017}
S.~Hencl and V.~Tengvall.
\newblock Sharpness of the differentiability almost everywhere and capacitary
  estimates for {S}obolev mappings.
\newblock {\em Rev. Mat. Iberoam.}, 33(2):595--622, 2017.

\bibitem{IwaOnn2009}
T.~Iwaniec and J.~Onninen.
\newblock Hyperelastic deformations of smallest total energy.
\newblock {\em Arch. Ration. Mech. Anal.}, 194(3):927--986, 2009.

\bibitem{IwaOnn2017}
T.~Iwaniec and J.~Onninen.
\newblock Limits of {S}obolev homeomorphisms.
\newblock {\em J. Eur. Math. Soc. (JEMS)}, 19(2):473--505, 2017.

\bibitem{IwaOnnZhu2020}
T.~Iwaniec, J.~Onninen, and Zh. Zhu.
\newblock Deformations of bi-conformal energy and a new characterization of
  quasiconformality.
\newblock {\em Arch. Ration. Mech. Anal.}, 236(3):1709--1737, 2020.

\bibitem{MarMiz1973}
M.~Marcus and V.~J. Mizel.
\newblock Transformations by functions in {S}obolev spaces and lower
  semicontinuity for parametric variational problems.
\newblock {\em Bull. Amer. Math. Soc.}, 79:790--795, 1973.

\bibitem{MolVod2020}
A.~Molchanova and S.~Vodop$'$yanov.
\newblock Injectivity almost everywhere and mappings with finite distortion in
  nonlinear elasticity.
\newblock {\em Calc. Var. Partial Differential Equations}, 59(1):Paper No. 17,
  25, 2020.

\bibitem{MosPas2014}
G.~Moscariello and A.~Passarelli~di Napoli.
\newblock The regularity of the inverses of {S}obolev homeomorphisms with
  finite distortion.
\newblock {\em J. Geom. Anal.}, 24(1):571--594, 2014.

\bibitem{MulSpe1995}
S.~M\"uller and S.~Spector.
\newblock An existence theory for nonlinear elasticity that allows for
  cavitation.
\newblock {\em Arch. Ration. Mech. Anal.}, 131(1):1--66, 1995.

\bibitem{Onn2000}
J.~Onninen.
\newblock Differentiability of monotone {S}obolev functions.
\newblock {\em Real Anal. Exchange}, 26(2):761--772, 2000/01.

\bibitem{Onn2006}
J.~Onninen.
\newblock Regularity of the inverse of spatial mappings with finite distortion.
\newblock {\em Calc. Var. Partial Differential Equations}, 26(3):331--341,
  2006.

\bibitem{Res1966}
Yu.~G. Reshetnyak.
\newblock Generalized derivatives and differentiability almost everywhere.
\newblock {\em Dokl. Akad. Nauk SSSR}, 170:1273--1275, 1966.

\bibitem{Resh1967-2}
Yu.~G. Reshetnyak.
\newblock Space mappings with bounded distortion.
\newblock {\em Sib. Math. J.}, 8(3):466--487, 1967.

\bibitem{Res1968}
Yu.~G. Reshetnyak.
\newblock Generalized derivatives and differentiability almost everywhere.
\newblock {\em Mat. Sb. (N.S.)}, 75(117):323--334, 1968.

\bibitem{Ste1923}
W.~Stepanoff.
\newblock \"{U}ber totale {D}ifferenzierbarkeit.
\newblock {\em Math. Ann.}, 90(3-4):318--320, 1923.

\bibitem{Ten2014}
V.~Tengvall.
\newblock Differentiability in the {S}obolev space {$W^{1,n-1}$}.
\newblock {\em Calc. Var. Partial Differential Equations}, 51(1--2):381--399,
  2014.

\bibitem{Vai1965-1}
J.~V\"{a}is\"{a}l\"{a}.
\newblock Two new characterizations for quasiconformality.
\newblock {\em Ann. Acad. Sci. Fenn. Ser. A I No.}, 362:12, 1965.

\bibitem{Vod2018}
S.~K. Vodop$'$yanov.
\newblock On the differentiability of mappings of {S}obolev class {$W^1_{n-1}$}
  with conditions on the distortion function.
\newblock {\em Sibirsk. Mat. Zh.}, 59(6):1240--1267, 2018.

\bibitem{Vod2020}
S.~K. Vodop$'$yanov.
\newblock The regularity of inverses to {S}obolev mappings and the theory of
  {$\mathcal{Q}_{q,p}$}-homeomorphisms.
\newblock {\em Siberian Mathematical Journal}, 61(6):1002--1038, 2020.

\bibitem{VodGold1976}
S.~K. Vodop$'$yanov and V.~M. Gol$'$dshtein.
\newblock Quasiconformal mappings and spaces of functions with generalized
  first derivatives.
\newblock {\em Sib. Math. J.}, 17(3):399--411, 1976.

\end{thebibliography}

\end{document}